\documentclass[11pt]{amsart}
\usepackage{amssymb}
\usepackage{latexsym,bm}

\numberwithin{equation}{section} \theoremstyle{section}
\newtheorem{De}[equation]{Definition}

\theoremstyle{plain}

\newtheorem{Prop}[equation]{Proposition}
\newtheorem{Theorem}[equation]{Theorem}

\newtheorem{Lemma}[equation]{Lemma}
\newtheorem{Cor}[equation]{Corollary}

\title{Crossed products by {\Large $\alpha$}-simple  automorphisms
on
$C$*-algebras C(X,A)}

\author{JIAJIE HUA}

\date{}

\begin{document}

\maketitle \markboth{JIAJIE HUA}{CROSSED PRODUCTS BY {\large
$\alpha$}-SIMPLE AUTOMORPHISMS ON $C$*-ALGEBRAS}
\renewcommand{\thefootnote}{{}}
\footnote{\hspace{-14pt} {\it 2000 Mathematics Subject Classification}. Primary 46L55: Secondary 46L35, 46L40.The author was supported the National
Natural Science Foundation of China (Nos. 10771069, 10671068,10771161).}
\begin{abstract}  Let $X$ be a Cantor set, and let $A$ be a
unital separable simple amenable $C$*-algebra with tracial rank zero which
satisfies the Universal Coefficient Theorem, we use $C(X,A)$ to denote the
set of all continuous functions from $X$ to $A$, let $\alpha$ be
an automorphism on $C(X,A)$. Suppose that $C(X,A)$ is
$\alpha$-simple and $[\alpha]=[\mbox{id}_{1\otimes A}]$ in $KL(1\otimes A,1\otimes A)$, we show that $C(X,A)\rtimes_{\alpha}\mathbb{Z}$ has tracial rank zero.

\end{abstract}

\section{Introduction}
Transformation group $C$*-algebras of minimal homeomorphisms of the
Cantor set are AT algebras (direct limits of circle algebras) with
real rank zero is implicit in Section 8 of \cite{R.Herman1}, with
the main step having been done in \cite{Putnam}. Elliott and Evans
proved in \cite{Elliott} that the irrational rotation algebras are
AT algebras with real rank zero, then Lin proved   that irrational higher dimensional noncommutative tori of
the form $C(T^k)\rtimes_\theta \mathbb{Z}$ are in fact AT algebras (\cite{H.Lin11}). Recently Phillips constructs
an inductive proof that every simple higher noncommutative torus is an AT algebra (\cite{Phillips1}). More generally,  Lin and Phillips proved the following result (\cite{H.Lin5}) : Let $X$ be an infinite compact metric space
with finite covering dimensional and let $\alpha$: $X\rightarrow X$
be a minimal homeomorphism,  the associated crossed product
$C$*-algebra $A=C(X)\rtimes_{\alpha}\mathbb{Z}$ has tracial rank
zero whenever the image of $K_0(A)$ in Aff($T(A)$) is dense.
In particular, these algebras all
belong to the class known currently to be classifiable by
K-theoretic invariants in the sense of the Elliott classification
program (\cite{Elliott1}).

 On the other hand,  Lin proved in \cite{H.Lin3} that let $A$ be a unital separable
simple $C$*-algebra with tracial rank zero and let $\alpha$ be an
automorphism on $A$.  Suppose that $\alpha$ has certain
Rokhlin property and there is an integer $J\geq1$
such that $[\alpha^J]=[\mbox{id}_A]$ in $KL(A,A)$. Then
$A\rtimes_{\alpha}\mathbb{Z}$ has tracial rank zero.

 In present, we don't know what happen when $C$*-algebras are nor commutative neither simple.
 In this paper, let $X$ be a Cantor set, let $A$ be a
unital separable simple amenable $C$*-algebra with tracial rank zero which
satisfies the Universal Coefficient Theorem, we consider the $C$*-algebra $C(X,A),$  all
continuous functions from $X$ to $A$. When $A$ is isomorphic to $\mathbb{C},$ it is just the case in \cite{Putnam}. When $C(X,A)$
 is not isomorphic to $\mathbb{C},$ $C(X,A)$ is
neither commutative nor simple, so it is different from above
 cases and contains Cantor set case. We prove the following result: let $X$ be a Cantor set, let $A$ be a
unital separable simple amenable $C$*-algebra with tracial rank zero which
satisfies the Universal Coefficient Theorem, and let $\alpha$ be an
automorphism on $C(X,A)$. Suppose that $C(X,A)$ is
$\alpha$-simple and $[\alpha]=[\mbox{id}_{1\otimes A}]$ in $KL(1\otimes A,1\otimes A)$, we present
a proof that $C(X,A)\rtimes_{\alpha}\mathbb{Z}$ has tracial rank
zero, therefore they are determined by their
graded ordered K-theory.

The condition $C(X,A)$ is $\alpha$-simple is necessary to guarantee that $C(X,A)\rtimes_{\alpha}\mathbb{Z}$ is simple, because in present classification theorems we have is mainly for simple $C$*-algebras. The second condition $[\alpha]=[\mbox{id}_{1\otimes A}]$ in $KL(1\otimes A,1\otimes A)$ means action of $\alpha$ take K-theory of each fiber of $X$ invariant. An early version of second condition is weaker and can not ensure approximate unitary equivalence of the action on fiber which was pointed us by Hiroki Matui.

This paper is organized as follows. In Section 2 we introduce
notation and give some elementary properties of $\alpha$-simple
automorphisms on $C$*-algebras $C(X,A)$. In Section 3, we prove
that, under our hypotheses, $C(X,A)\rtimes_{\alpha}\mathbb{Z}$  has
tracial rank zero.

\section{Notation and $\alpha$-simple {\footnotesize $C$}*-algebras}

We will use the following convention:

(1) Let $A$ be a $C$*-algebra, let $a\in A$ be a positive element
and let $p\in A$ be a projection. We write $[p]\leq [a]$ if there is
a projection $q\in \overline{aAa}$ and a partial isometry $v\in A$
such that $v^{*}v=p$ and $vv^{*}=q.$

(2) Let $A$ be a $C$*-algebra. We denote by Aut$(A)$ the
automorphism group of $A$. If $A$ is unital and $u\in A$ is a
unitary, we denote by ad$u$ the inner automorphism defined by
ad$u(a)=u^{*}au$ for all $a\in A.$

(3) Let $x\in A$, $\varepsilon>0$ and $\mathcal{F}\subset A.$ We
write $x\in_{\varepsilon}\mathcal{F},$ if
dist$(x,\mathcal{F})<\varepsilon$, or there is $y\in \mathcal{F}$
such that $\|x-y\|<\varepsilon.$

(4) Let $A$ be a $C$*-algebra and $\alpha\in$ Aut$(A)$. We say $A$
is $\alpha$-simple if $A$ does not have any non-trivial
$\alpha$-invariant closed two-sided ideals.

(5) Let $A$ be a unital $C$*-algebra and $T(A)$ the compact convex
set of tracial states of $A$. If $\alpha$ is an automorphism of $A$, we use
$T^{\alpha}(A)$ to denote the $\alpha$-invariant tracial states, which
is again a compact convex set. We define an affine mapping $r$ of
$T(A \rtimes_{\alpha}\mathbb{Z})$ into $T^{\alpha}(A)$ by the
restriction $r(\tau)=\tau|A.$

We recall the definition of tracial topological rank of
$C$*-algebras.
\begin{De} \cite{H.Lin1} Let $A$ be a unital simple $C$*-algebra.
Then $A$ is said to have tracial (topological) rank zero if for any $\varepsilon>0$, any
finite set $\mathcal{F}\subset A$  and any
nonzero positive element $a\in A,$ there exists a finite dimensional
$C$*-subalgebra $B\subset A$ with id$_{B}=p$ such
that:\\
(1) $\|px-xp\|<\varepsilon$ for all $x\in \mathcal{F}$.\\
(2) $pap\in_{\varepsilon} B$ for all $x\in \mathcal{F}$.\\
(3) $[1-p]\leq [a].$
\end{De}

If $A$ has tracial rank zero, we write $\mathrm{TR}(A)=0.$

\begin{De} Let $X$ be a compact metric space and let $A$ be a
 $C$*-algebra, we say a map $\beta:X\rightarrow$ Aut$(A)$
is strongly continuous if for any $\{x_n\}$ with $d(x_n,x)\rightarrow 0$ when $n\rightarrow\infty$, we have $\|\beta_{x_n}(a)-\beta_x(a)\|\rightarrow
0$ for all $a\in A.$

\end{De}

\begin{Lemma}Let $X$ be a compact metric space, let $A$ be a
unital simple $C$*-algebra and $\alpha\in$Aut$(C(X,A))$. Then
$C(X,A)$ is $\alpha$-simple if and only if there is a minimal
homeomorphism $\sigma$ from $X$ to $X$ and there is a strongly
continuous map $\beta$ from $X$ to Aut$(A)$, denote by $x$ to $\beta_x$, such
that $\alpha(f)(x)=\beta_{\sigma^{-1}(x)}(f(\sigma^{-1}(x)))$.
\end{Lemma}
\begin{proof}
We firstly prove that the set consists of  all ideals  of $C(X,A)$
is same as the one consists of $\{f\in C(X,A)| f(x)=0  \mbox{ for
any } x\in Y\}$ for all closed subset $Y\subset X.$

 Clearly, $\{f\in C(X,A)| f(x)=0  \mbox{ for
any } x\in Y\}$ is an ideal of $C(X,A)$ for any closed subset
$Y\subset X.$

Conversely, let $I$ be an ideal of $C(X,A)$, and let $X_I=\cap_{f\in
I}\{x\in X |f(x)=0\}$ be a closed subset of $X$, We claim $I=\{f\in
C(X,A)| f(x)=0 \mbox{ for any } x\in X_I\}.$ Clearly,
$I\subseteq\{f\in C(X,A)| f(x)=0 \mbox{ for any } x\in X_I\}.$

In the following we will prove $\{f\in C(X,A)| f(x)=0 \mbox{ for any } x\in
X_I\}\subseteq I.$

For any $\varepsilon>0$ and  any $x_0\notin X_I$, there is a $f\in I$
such that $f(x_0)\neq 0$, thus for any $h\in \{f\in C(X,A)| f(x)=0 \mbox{ for any } x\in X_I\}$,
 since $A$ is a unital simple $C$*-algebra, by Lemma 3.3.6 of \cite{H.Lin2}, there
 exists $a_i\in A,i=1,\dots,k$ such that $h(x_0)=\sum_{i=1}^{k}a_i
 f(x_0)a_i^*$. Let $g(x)=\sum_{i=1}^{k}a_i
 f(x)a_i^*$, since $f(x)\in I,$ we have $g(x)\in I,$
 then  there is a neighborhood $O(x_0)$ of
 $x_0$ such that $\|g(y)-h(y)\|< \varepsilon$ for all $y\in O(x_0)$. Since $X$ is compact,
 there exist $g_i(x)\in I$ and $x_i\in X,i=1,\cdots,n,$ such that $\{O(x_i)\}_{i=1}^n$ cover
 $X$ and $\|g_i(x)-h(x)\|<\varepsilon$ for all $x\in O(x_i)$.

 Let $\{f_1,\cdots,f_n\}$ be a subset of nonnegative functions
 in $C(X)$(a partition of the unit) satisfying the following
 conditions:

 (1) $f_i(x)=0$ for all $x\notin O(x_i), i=1,2,\cdots,n$ and

 (2) $\sum_{i=1}^n f_i(x)=1$ for all $x\in X.$

 We compute
 $$\|\sum_{i=1}^n g_i(x)
 f_i(x)-h(x)\|\leq\sum_{i=1}^n\|g_i(x)-h(x)\||f_i(x)|
 <\varepsilon,$$
 then we have $h\in I$, so $\{f\in C(X,A)| f(x)=0 \mbox{ for any } x\in X_I\}\subseteq
 I.$ The claim follows.

From above, we know that the set of all
the maximal ideals of $C(X,A)$ is  $\{I_{x}\}_{x\in X},$ where
$I_{x}=\{f\in C(X,A) | f(x)=0\}.$  Given  $x\in X$, because $\alpha$
is an automorphism, there is $y\in X$ such that $\alpha(I_x)=I_y.$
We define  $\sigma: X\rightarrow X$ by $x$ to $y$ and
$\beta_{\sigma(x)}:A\rightarrow A$ by
$\alpha(C(X,A)/I_{x})=C(X,A)/I_{\sigma{(x)}},
\beta_{x}(a)=\alpha(a)(\sigma(x)).$ For any fixed $y\in X$ and $f\in
C(X,A),$ then
$\alpha(f)(y)=\alpha((f-f(\sigma^{-1}(y)))+f(\sigma^{-1}(y)))(y)=\alpha(f(\sigma^{-1}(y)))(y).$

So $\alpha(f){(y)}=\beta_{\sigma^{-1}(y)}(f(\sigma^{-1}(y)))$ for
all $y\in X.$

For any $a\in A$, if $d(x_n-x)\rightarrow 0$ when $n\rightarrow
\infty$, we have
$\|\beta_{x_n}(a)-\beta_x(a)\|=\|\alpha(a)(\sigma(x_n))-\alpha(a)(\sigma(x))\|\rightarrow
0$ when $n\rightarrow \infty$.

 For any $a,b\in A$ and $x\in X$
$$\beta_{x}(a+b)=\alpha(a+b)(\sigma(x))=\alpha(a)(\sigma(x))+\alpha(b)(\sigma(x))=\beta_{x}(a)+\beta_{x}(b),$$
$$\beta_{x}(a\cdot b)=\alpha(a\cdot b)(\sigma(x))=\alpha(a)(\sigma(x))\cdot\alpha(b)(\sigma(x))=\beta_{x}(a)\cdot\beta_{x}(b),$$
$$\beta_{x}(a^*)=\alpha(a^*)(\sigma(x))=\alpha(a)^*(\sigma(x))=\beta_{x}(a)^*,$$
so $\beta_{x}$ is a homomorphism,

Since
$$\beta_{x}(\alpha^{-1}(b)(x))=\beta_{x}((\alpha^{-1}(b)(x)-\alpha^{-1}(b))+\alpha^{-1}(b))=\alpha(\alpha^{-1}(b))(\sigma(x))=b$$
and $A$ is simple, we have $\beta_{x}$ is an automorphism of $A$ for
any $x\in X$.

If there exist $x_1,x_2\in X$, $x_1\neq x_2$ such that
$\sigma(x_1)=\sigma(x_2),$ it is easy to find $f\in C(X,A)$ such
that $f(x_1)=1,f(x_2)=0.$ So
$$\alpha(f)(\sigma(x_1))=\beta_{x_1}(f(x_1))=1$$
$$\alpha(f)(\sigma(x_2))=\beta_{x_2}(f(x_2))=0$$
this is impossible, so $\sigma$ is injective.

By considering $\alpha^{-1},$ we can also get $\sigma$ is
surjective.

When $n\rightarrow \infty,$ if $d(x_n,x)\rightarrow 0,$ then for any
$f\in C(X,A),$
\begin{eqnarray*}
&&\|f(\sigma(x_{n}))-f(\sigma(x))\|\\&=&\|\alpha(\alpha^{-1}(f))(\sigma(x_n))-\alpha(\alpha^{-1}(f))(\sigma(x))\|\\
&=&\|\beta_{x_n}(\alpha^{-1}(f)(x_n))-\beta_{x}(\alpha^{-1}(f)(x))\|\\&=&\|\beta_{x_n}(\alpha^{-1}(f)(x_n)-\alpha^{-1}(f)(x))+
\beta_{x_n}(\alpha^{-1}(f)(x))-\beta_{x}(\alpha^{-1}(f)(x))\|\\
&&\rightarrow 0.\end{eqnarray*} so $\sigma$ is continuous, we can
also get $\sigma^{-1}$ is continuous by considering $\alpha^{-1},$
we omit it. so $\sigma$ is a homeomorphism.

Since $C(X,A)$ is $\alpha$-simple, it does not have any non-trivial
$\alpha$-invariant closed two-sided ideals of $C(X,A)$, then we have
not any non-trivial $\sigma$-invariant closed subset of $X$, so
$\sigma$ is a minimal homeomorphism of $X$.

Conversely, if $\sigma$ is a minimal homeomorphism of $X$ and
$\beta: X \rightarrow Aut(A)$ is a strongly continuous map, define
$\alpha(f){(y)}=\beta_{\sigma^{-1}(y)}f(\sigma^{-1}(y))$ for all
$y\in X,$ it is easy to verify $\alpha\in Aut(C(X,A))$ and $C(X,A)$ is
$\alpha$-simple.

\end{proof}

\begin{Lemma}Let $X$ be an infinite compact metric space, let $A$ be a
unital simple $C$*-algebra and  $\alpha\in$Aut$(C(X,A))$. Then
$C(X,A)$ is $\alpha$-simple if and only if the crossed product
$C(X,A)\rtimes_{\alpha} \mathbb{Z}$ is simple.
\end{Lemma}
\begin{proof} Let $I$ be an $\alpha$-invariant norm closed
two-sided ideal of $C(X,A)$. Then $I\rtimes_{\alpha}\mathbb{Z}$ is a
norm closed two-sided ideal of $C(X,A)\rtimes_{\alpha} \mathbb{Z}$
by Lemma 1 of \cite{Jang}.

Conversely, for any positive element $f$ of the $C$*-algebra
$C(X,A)$, any finite set $\mathcal{F}=\{f_{i};i=1,2,\cdots.n\}
\subset C(X,A)$, any $s_{i}\subset \mathbb{N},i=1,2,\cdots,n$, and
any $\varepsilon>0,$ we claim that  there exists a positive element
$g\in C(X,A)$ with $\|g\|=1$ such that
$$\|gfg\|\geq\|f\|-\varepsilon, \quad
\|gf_{i}\alpha^{s_{i}}(g)\|\leq \varepsilon, \quad
i=1,2,\cdots,n.\quad\quad (*)$$

Because $X$ is a compact set, we can get a point $x\in X$ such that
$\|f(x)\|=\|f\|.$ By Lemma 2.3  we  get a minimal homeomorphism $\sigma$ of X
. Since $\sigma$ is the minimal homeomorphism of $X$,
there exists a neighborhood $O(x)$ of $x$ such that
$\sigma^{i}(O(x))$ are disjoint for $i=1,\cdots,s,$ where
$s=\max\{s_{i},i=1,2,\cdots,n\}$ and $\|f(y)-f(x)\|<\varepsilon$ for
all $y\in O(x).$ It is easy to find a continuous function $g$ of $X$
to $[0,1]$ such that $g(x)=1$ and $g(z)=0$ for all $z\notin O(x).$
Then $g$ satisfies the conditions $(*)$. The claim follows.

Use the condition $(*)$ and $C(X,A)$ is $\alpha$-simple, we can
complete the proof as same as Theorem 3.1 of \cite{A.Kishimoto1}, we
omit it.
\end{proof}

Let $u$ be the unitary implementing the action of $\alpha$ in the
transformation group $C$*-algebra
$C(X,A)\rtimes_{\alpha}\mathbb{Z}$, then $ufu^*=\alpha(f).$ For a
nonempty closed subset $Y\subset X,$ we define the $C$*-subalgebra
$B_Y$ to be
$$B_Y=C^*(C(X,A),uC_0(X\backslash Y,A))\subset
C(X,A)\rtimes_{\alpha}\mathbb{Z}.$$ We will often let $B$ denote the
transformation group $C$*-algebra
$C(X,A)\rtimes_{\alpha}\mathbb{Z}$. If $Y_{1}\supset
Y_2\supset\cdots$ is a decreasing sequence of closed subsets of $X$
with $\cap_{n=1}^{\infty}Y_n=\{y\},$ then $B_{\{y\}}=\lim
B_{Y_{n}}.$

Let $Y\subset X$, and let $x\in Y.$ If $C(X,A)$ is $\alpha$-simple, by Lemma 2.3 we have a minimal homeomorphism $\sigma$ of $X.$ The first return time
$\lambda_Y(x)$ (or $\lambda(x)$ if $Y$ is understood) of $x$ to $Y$
is the smallest integer $n\geq 1$ such that $\sigma^n(x)\in Y$.

The following result is well known in the area, and is easily
proved:

 \begin{Lemma}If
$Y$ is a nonempty clopen subset and $\sigma$ is a minimal
homeomorphism of $X$. Then $\sup_{x\in Y}\{\lambda_Y(x)\}<\infty.$
\end{Lemma}

Let $Y\subset X$ is a nonempty clopen subset. Let
$n(0)<n(1)<\cdots<n(l)$ be the distinct values of $\lambda(x)$ for
$x\in Y$. The Rokhlin tower based on a subset $Y\subset X$ with
$Y\neq \emptyset$ consist of the partition
$$Y=\coprod_{k=0}^{l}\{x\in Y: \lambda(x)=n(k)\}$$
of  $Y$ (the sets here are the base sets), and the corresponding
partition
$$X=\coprod_{k=0}^{l}\coprod_{j=0}^{n(k)-1}\sigma^j(\{x\in Y: \lambda(x)=n(k)\})$$
of $X$.

Actually, for our purposes it is more convenient to use the
partition

$$X=\coprod_{k=0}^{l}\coprod_{j=1}^{n(k)}\sigma^j(\{x\in Y: \lambda(x)=n(k)\}).$$

Note that
$$Y=\coprod_{k=0}^{l}\sigma^{n(k)}(\{x\in Y: \lambda(x)=n(k)\})$$

Since $Y$ is both closed and open, the sets $$Y_k=\{x\in Y:
\lambda(x)=n(k)\}$$ are all closed and open.

\begin{Prop} Let $X$ be a Cantor set, let $A$ be a unital simple
$C$*-algebra and $\alpha\in Aut(C(X,A))$. Suppose $C(X,A)$ is
$\alpha$-simple, let $Y\subset X$ be a nonempty clopen subset. Then
there exists a unique isomorphism
$$\gamma_Y: B_Y\rightarrow\bigoplus_{k=0}^{l}C(Y_k, M_{n(k)}(A))$$
such that if $f\in C(X,A),$ then
$$\gamma_Y(f)_k=\mbox{diag}(\alpha^{-1}(f)|_{Y_k},\alpha^{-2}(f)|_{Y_k},\cdots,\alpha^{-n(k)}(f)|_{Y_k})$$
and if $f\in C_0(X\backslash Y,A),$ then $$(\gamma_Y(uf))_k=s_k
\gamma_Y (f)_k$$ where $s_k\in M_{n(k)}\subset C(Y_k,M_{n(k)}(A))$
is defined by $$s_k=\left(\begin{matrix}0&0&0&\cdots&\cdots&0&0&1\\
1&0&0&\cdots&\cdots&0&0&0\\
0&1&0&\cdots&\cdots&0&0&0\\
\vdots&\vdots&\vdots&\ddots&&\vdots&\vdots&\vdots\\
\vdots&\vdots&\vdots&&\ddots&\vdots&\vdots&\vdots\\
0&0&0&\cdots&\cdots&0&0&0\\
0&0&0&\cdots&\cdots&0&1&0\\
 \end{matrix}\right)$$
\end{Prop}
\begin{proof}By Lemma 2.3, there is a minimal homomorphism
$\sigma$ from $X$ to $X$ and there is a strongly continuous map from $X$
to Aut$(A)$, denote by $x$ to $\beta_x$, such that
$\alpha(f)(x)=\beta_{\sigma^{-1}(x)}(f(\sigma^{-1}(x)))$.

From above, we have $$Y=\coprod_{k=0}^{l}\sigma^{n(k)}(\{x\in Y:
\lambda(x)=n(k)\})=\coprod_{k=0}^{l}\sigma^{n(k)}(Y_k).$$

It is easy to verify that the $\gamma_Y$ is a homomorphism.

Since $\alpha(f)(x)=\beta_{\sigma^{-1}(x)}(f(\sigma^{-1}(x)))$, it
is easy to see that $\gamma_Y$ is injective and surjective.
\end{proof}

\section{Main result}

\begin{Lemma} Let $X$ be an infinite compact metric space, let $A$ be a
unital simple $C$*-algebra and $\alpha\in$Aut$(C(X,A)).$ If $C(X,A)$
is $\alpha$-simple, let $y\in X$, then $B_{\{y\}}$ is simple.
\end{Lemma}
\begin{proof} Let $I\subset B_{\{y\}}$ be a nonzero ideal.  Since
$X$ is an infinite compact metric space, $A$ is a unital simple
$C$*-algebra and $C(X,A)$ is $\alpha$-simple, we have a minimal
homeomorphism  $\sigma$ of $X$ by Lemma 2.3. Then $I\cap C(X,A)$ is
an ideal in $C(X,A).$ So we can write $I\cap C(X,A)=C_0(U,A)$ for
some open set $U\subset X$ by the proof of Lemma 2.3, which is necessarily given by,
$$U=\{x\in X: \mbox{there is } f\in I\cap C(X,A) \mbox{ such that } f(x)\neq 0\}.$$

We first claim that $U\neq \emptyset,$

Write
$$B_{\{y\}}=\lim_{\rightarrow}B_{Y_m}$$ for some decreasing sequence $Y_{1}\supset
Y_2\supset\cdots$ of closed subsets of $X$ with
$\cap_{n=1}^{\infty}Y_n=\{y\},$ and $\mbox{int}(Y_m)\neq \emptyset,$

Then there exists $m$ such that $B_{Y_m}\cap I\neq \{0\}$ by Lemma 3.5.10 of \cite{H.Lin2}. Let $a$
be a nonzero element of this intersection. Using Proposition 2.6,
one can fairly easily prove that there is $N$ such that every
element of  $B_{Y_m}$ can be written in the form $\sum_{n=-N}^{N}
f_n u^n,$ with $f_n\in C(X,A)$ for $-N\leq n\leq N.$ Moreover, if
$a\neq 0,$ and one writes $a^*a=\sum_{n=-N}^{N} f_n u^n,$ then
$f_0\neq 0.$ Choose $x\in X$ such that $f_0(x)\neq 0,$  choose a
neighborhood $V$ of $x$ such that the sets $h^n(V)$, for $-N\leq n
\leq N$, are disjoint, and choose $g\in C(X)$ such that supp$(g)\in
V$ and $g(x)\neq 0.$ Then $g\in B_{\{y\}},$ and one checks that

$$ga^*ag=\sum_{n=-N}^N gf_n u^n g=\sum_{n=-N}^N g(g\circ \sigma^n)f_n u^n=g^2f_0.$$
So $g^2f_0$ is a nonzero element of $I\cap C(X,A)$, proving the
claim.

We next claim that $\sigma^{-1}(U\backslash\{\sigma(y)\})\subset U.$
So let $x\in U\backslash\{\sigma(y)\}.$ Choose $f\in I\cap C(X,A)$
such that $f(x)\neq 0,$ and choose $g\in C_0(X\backslash \{y\})$
such that $g(\sigma^{-1}(x))\neq 0.$ Then $ug\in B_{\{y\}},$ and
$$(ug)^*f(ug)=\overline{g}u^*fug=|g|^2(u^*fu).$$
Thus $|g|^2(u^*fu)\in I\cap C(X,A)$ and is
nonzero at $\sigma^{-1}(x)$, This proves the claim.

We further claim that $\sigma(U\backslash\{y\})\subset U.$ The proof
is similar: let $x\in U\backslash \{y\},$ let $f\in I\cap C(X,A)$
and $g\in C_0 (X\backslash \{y\})$ be nonzero at $x,$ and consider
$ug\in B_{\{y\}},$
$$(ug)f(ug)^*=ugf\overline{g}u^*=|g\circ\sigma^{-1}|^2 (ufu^*).$$
Thus $|g\circ\sigma^{-1}|^2 (ufu^*)\in I\cap
C(X,A)$ and is nonzero at $\sigma(x)$, So
$\sigma(U\backslash\{y\})\subset U.$

Now set $Z=X\backslash U.$ The last two claims above imply that if
$x\in X$ and $x$ is not in the orbit of $y$, Then $\sigma^k(x)\in Z$
for all $k\in \mathbb{Z}.$ Since $\sigma$ is minimal, $Z$ is closed,
and $Z\neq X$, this is impossible. If $\sigma^n(y)\in Z$ for some
$n>0,$ then $\sigma^{-1}(U\backslash\{\sigma(y)\})\subset
U\backslash \{\sigma(y)\}$ implies $h^{k}(y)\in Z$ for all $k\geq
n.$ Since $\sigma$ is dense by minimality, this is also a
contradiction. Similarly, if $\sigma^n(y)\in Z$ for some $n\leq 0,$
then $Z$ would contain the dense set $\{\sigma^{k}(y):k\leq n\},$
again a contradiction.

So $U=X$ and $1\in I.$ Then $B_{\{y\}}$ is simple.

\end{proof}

\begin{Lemma}Let $X$ be a Cantor set, let $A$ be a
unital separable simple amenable $C$*-algebra with tracial rank zero which
satisfies the UCT (Universal Coefficient Theorem), and let $\alpha\in
Aut(C(X,A))$. Suppose $C(X,A)$ is $\alpha$-simple.  It follows that
for any $y\in X,$ the $C$*-algebra $B_{\{y\}}$ has tracial rank
zero.
\end{Lemma}
\begin{proof} Let $Y\subset X$ be a nonempty clopen subset,
applying Proposition 2.6, we have $$B_Y\cong\bigoplus_{k=0}^{l}C(Y_k,
M_{n(k)}(A))=\bigoplus_{k=0}^{l}C(Y_k)\otimes M_{n(k)}(A).$$

Let $y\in X,$ then there exists
  a decreasing sequence $Z_{1}\supset Z_2\supset\cdots$ of clopen
subsets of $X$ with $\cap_{n=1}^{\infty}Z_n=\{y\}.$ Since the
property of having locally finite decomposition rank passes to
tensor products, inductive limits by Proposition 1.3 of
\cite{w.winter2}, then $B_{\{y\}}=\lim B_{Z_{n}}$ has locally finite
decomposition rank. Clearly, for any $n\in \mathbb{N}$, $B_{Z_{n}}$
has real rank zero, so $B_{\{y\}}$ has also real rank zero. Since
$A$ is a unital separable simple amenable $C$*-algebra with tracial rank zero
which satisfies the Universal Coefficient Theorem, $A$ is an
AH-algebra of slow dimension growth with real rank zero by
\cite{H.Lin10}, then $A$ is an AH-algebra of bounded dimension growth
by \cite{Dadarlat}. By Corollary 3.1 of \cite{Toms1}, $A$ absorbs
the Jiang-Su algebra $\mathcal{Z}$ tensorially. By Corollary 3.4 of
\cite{Toms2},  $B_{\{y\}}$ absorbs the Jiang-Su algebra
$\mathcal{Z}$ tensorially, so the $C$*-algebra $B_{\{y\}}$ has
tracial rank zero by Theorem 2.1 of \cite{w.winter2}.

\end{proof}

\begin{Lemma} Any trace on $B_{\{y\}}$ is restricted to $C(X)$ is a
$\sigma$-invariant measure on $X.$
\end{Lemma}
\begin{proof}Let $\tau$ be a normalized trace on $B_{\{y\}}$, and let
$f\in C(X).$ Set
$$a=u|f-f(y)|^{1/2}\quad\mbox{and}\quad b=u(f-f(y)\cdot 1)|f-f(y)|^{-1/2}$$
Then $a$ and $b$ are both in $B_{\{y\}}$. Moreover,
$$\tau(f\circ\sigma^{-1}-f(y)\cdot 1)=\tau(u(f-f(y)\cdot 1)u^*)=\tau(ab^*)=\tau(b^*a)=\tau(f-f(y)\cdot 1).$$
Cancelling $\tau(f(y)\cdot 1),$ we get
$\tau(f\circ\sigma^{-1})=\tau(f).$
\end{proof}

\begin{Lemma} Let $X$ be a Cantor set, let $A$ be a
unital separable simple amenable $C$*-algebra with tracial rank zero which
satisfies the UCT, and let $\alpha\in Aut(C(X,A))$. Suppose $C(X,A)$
is $\alpha$-simple and $[\alpha]=[\mbox{id}_{1\otimes A}]$ in $KL(1\otimes A,1\otimes A)$. Let
$B=C(X,A)\rtimes_{\alpha}\mathbb{Z}$, and let $y\in X.$ Then for any
$\varepsilon>0$ and finite subset $\mathcal{F}\subset B,$ there is a
projection $p\in B_{\{y\}}$ such that:

(1) $\|pa-ap\|<\varepsilon$ for all $a\in \mathcal{F}.$

(2) $pap\in pB_{\{y\}}p$ for all $a\in \mathcal{F}.$

(3) $\tau(1-p)<\varepsilon$ for all $\tau\in T(B_{\{y\}}).$
\end{Lemma}
\begin{proof} We may assume that $\mathcal{F}=\mathcal{G}\cup \{u\}$
for some finite subset $\mathcal{G}\subset C(X,A).$

By Lemma 2.3, there is a minimal homomorphism $\sigma$ from $X$ to $X$
and there is a strongly continuous map from $X$ to Aut$(A)$, denote by
$x$ to $\beta_x$, such that
$\alpha(f)(x)=\beta_{\sigma^{-1}(x)}(f(\sigma^{-1}(x)))$.

Choose $N_0\in \mathbb{N}$ so large that $4\pi /N_0<\varepsilon.$
Choose $\delta_0>0$ with $\delta_0<\frac{1}{2}\varepsilon$ and so
small that $d(x_1,x_2)<4\delta_0$ implies
$\|f(x_1)-f(x_2)\|<\frac{1}{4}\varepsilon$ for all $f\in
\bigcup_{i=0}^{N_0}\alpha^{-i}(\mathcal{G}).$ Choose $\delta>0$ with
$\delta\leq\delta_0$ and such that whenever $d(x_1,x_2)<\delta$ and
$0\leq n\leq N_0,$ then
$d(\sigma^{-n}(x_1),\sigma^{-n}(x_2))<\delta_0.$

Since $\sigma$ is minimal, there is $N>N_0+1$ such that
$d(\sigma^N(y),y)<\delta.$ Since $\sigma$ is free, there is a clopen
neighborhood $Y$ of $y$ in $X$ such that
$$\sigma^{-N_0}(Y),\sigma^{-N_0+1}(Y),\dots, Y, \sigma(Y), \dots,
\sigma^N(Y)$$ are disjoint  and  all have diameter less than
$\delta$, and furthermore $\mu(Y)<\varepsilon/(N+N_0+1)$ for every
$\sigma$-invariant Borel probability measure $\mu$.

Define continuous functions projection
$q_0(x)=\left\{\begin{matrix}1,& x\in Y\\
0,&\quad x\in X\backslash Y
\end{matrix}\right.$

For $-N_0\leq n\leq N$ set, $q_n=\alpha^n(q_0)=\left\{\begin{matrix}1,& x\in \sigma^n(Y)\\
0,&\quad x\in X\backslash \sigma^n(Y)
\end{matrix}\right.$, so the $q_n$ are
mutually orthogonal projections in $B_{\{y\}}.$

We now have a sequence of projections:
$$q_{-N_0},\dots,q_{-1}, q_{0}, \dots, q_{N-N_0}, \dots, q_{N-1},
q_{N}.$$

The projections $q_0$ and $q_N$ live over clopen sets which are
disjoint but close to each other, and similarly for the pairs
$q_{-1}$ and $q_{N-1}$ down to $q_{-N_0}$ and  $q_{N-N_0}.$ We are
now going to use Berg's technique \cite{Berg} to splice this
sequence along the pairs of indices $(-N_0, N-N_0)$ through $(0,N),$
obtaining a loop of length $N$ on which conjugation by $u$ is
approximately the cyclic shift.

We claim that there is a partial isometry $w\in B_{\{y\}}$ such that
$w^*w=q_0$, $ww^*=q_N$ and
$\|wf|_Y-f|_{\sigma^N(Y)}w\|<\frac{\varepsilon}{4}$ for all $f\in
\bigcup_{i=0}^{N_0}\alpha^{-i}(\mathcal{G}).$

Let $x\in Y.$ The first return time $\lambda_Y(x)$ (or $\lambda(x)$
if $Y$ is understood) of $x$ to $Y$ is the smallest integer $n\geq
1$ such that $\sigma^n(x)\in Y$. By Lemma 2.5, we let
$n(0)<n(1)<\cdots<n(l)$ be the distinct values of $\lambda(x)$ for
$x\in Y$.

 We denote
$Y(k,j)=\sigma^j(\lambda^{-1}(n(k)))$,
 So
$$X=\coprod_{k=0}^{l}\coprod_{j=1}^{n(k)}\sigma^j(\{x\in Y: \lambda(x)=n(k)\})=
\coprod_{k=0}^{l}\coprod_{j=1}^{n(k)}Y(k,j).$$

Define continuous functions projection
$\chi_{Y(k,j)}=\left\{\begin{matrix}1,& x\in Y(k,j)\\
0,&\quad x\in X\backslash Y(k,j)
\end{matrix}\right.$

Define $w'=\sum_{k=0}^{l}\chi_{Y(k,N)}u^{N-n(k)}$, then $w'\in
B_{\{y\}}.$

$$w'^*w'=(\sum_{k=0}^{l}u^{-N+n(k)}\chi_{Y(k,N)})(\sum_{k=0}^{l}\chi_{Y(k,N)}u^{N-n(k)})=\sum_{k=0}^{l}\chi_{Y(k,n(k))}=q_0.$$

$$w'w'^*=(\sum_{k=0}^{l}\chi_{Y(k,N)}u^{N-n(k)})(\sum_{k=0}^{l}u^{-N+n(k)}\chi_{Y(k,N)})=\sum_{k=0}^{l}\chi_{Y(k,N)}=q_N.$$

Let $a|_Y=\left\{\begin{matrix}a, &\quad x\in Y\\
0, &\quad x\in X\backslash Y\end{matrix}\right.$ for all $a\in A.$

Since $A$ is a unital separable simple amenable $C$*-algebra with tracial
rank zero which satisfies the UCT, by the classification theorem of \cite{H.Lin7}, \cite{H.Lin10} and \cite{Elliott2}, $A$ is a unital separable simple AH-algebra.
By the assumption  $[\alpha]=[\mbox{id}_{1\otimes A}]$ in $KL(1\otimes A,1\otimes A)$ and by applying Lemma 4.1 of \cite{H.Lin3}, we have  $\tau(\beta_x(a))=\tau(a)$ for all $a\in A$ and for all $\tau\in A.$ Since TR$(A)=0$ and $[\alpha]=[\mbox{id}_{1\otimes A}]$ in $KL(1\otimes A,1\otimes A)$, $\tau(\beta_x(a))=\tau(a)$ for all $\tau\in A$ and for all $\tau\in A,$ by applying 3.6 of \cite{H.Lin12}, there is a clopen
neighborhood $Y_x$ of $x$, there exists an unitary $u_j'\in A$ such that
$u_{j}'|_{\sigma^j(Y_x)}u^ja|_{Y_x}u^{*j}u'^*_{j}|_{\sigma^j(Y_x)}\approx_{\frac{\varepsilon}{8}}a|_{\sigma^j(Y_x)}$
for all $a\in \{f(y)|f\in \cup_{i=0}^{N_0}\alpha^{-i}(\mathcal{G})\}.$ Because $Y$ is compact subset, we can get
$u_j\in C(X,A)$ such that
$u_{j}|_{\sigma^j(Y)}u^ja|_{Y}u^{*j}u^*_{j}|_{\sigma^j(Y)}\approx_{\frac{\varepsilon}{8}}a|_{\sigma^j(Y)}$
for all $a\in \{f(y)|f\in \cup_{i=0}^{N_0}\alpha^{-i}(\mathcal{G})\}.$

Define $w=(\sum_{k=0}^{l}u_{N-n(k)}\chi_{Y(k,N)})w',$ then $w\in
B_{\{y\}},w^*w=q_0, ww^*=q_N$ and
$$wa|_Yw^*=\sum_{k={0}}^lu_{N-n(k)} \chi_{Y(k,N)} u^{N-n(k)}a|_Yu^{-N+n(k)}\chi_{Y(k,N)} u^{*}_{N-n(k)}\approx_{\frac{\varepsilon}{8}}a|_{\sigma^N(Y)},$$
$wa|_Y\approx_{\frac{\varepsilon}{8}}a|_{\sigma^N(Y)}w,$ so
$\|wf|_Y-f|_{\sigma^N(Y)}w\|<\frac{\varepsilon}{4}$ for all $f\in
\bigcup_{i=0}^{N_0}\alpha^{-i}(\mathcal{G}).$ The claim follows.

For $t\in \mathbb{R}$ define  $v(t)=\cos(\pi t/2)(q_0+q_N)+\sin(\pi
t/2)(w-w^*).$ Then $v(t)$ is a unitary in the corner
$(q_0+q_N)B_{\{y\}}(q_0+q_N)$ whose matrix with respect to the
obvious block decomposition is

\begin{center}$v(t)=\bigg(\begin{matrix}\cos(\pi t/2)&-\sin(\pi t/2)\\
\sin(\pi t/2)&\cos(\pi t/2)
\end{matrix}\bigg)$
\end{center}
So
$\|v(t)(f|_Y+f|_{\sigma^N(Y)})-(f|_Y+f|_{\sigma^N(Y)})v(t)\|=\|((\cos(\pi
t/2)(q_0+q_N)+\sin(\pi
t/2)(w-w^*))(f|_Y+f|_{\sigma^N(Y)})-(f|_Y+f|_{\sigma^N(Y)})(\cos(\pi
t/2)(q_0+q_N)+\sin(\pi t/2)(w-w^*))\|=\|\sin(\pi
t/2)(wf|_Y-w^*f|_{\sigma^N(Y)})-\sin(\pi
t/2)(f|_{\sigma^N(Y)}w-f|_Yw^*)\|<\frac{\varepsilon}{2}$ for all
$f\in \bigcup_{i=0}^{N_0}\alpha^{-i}(\mathcal{G}).$ \hfill $(**)$

 For $0\leq k \leq N_0$ define $w_{k}=u^{-k}v(k/N_0)u^k,$ so
$w_{k}\in (q_{-k}+q_{N-k})B_{\{y\}}(q_{-k}+q_{N-k}).$

$\|uw_{k+1}u^*-w_{k}\|=\|v(k/N_0)-v((k-1))/N_0\|\leq 2\pi/
N_0<\frac{1}{2}\varepsilon$.

Now define $e_n=q_n$ for $0\leq n \leq N-N_0$, and for $N-N_0\leq
n\leq N$ write $k=N-n$ and set $e_n = w_k q_{-k}w_{k}^*.$ The two
definitions for $n=N-N_0$ agree because
$w_{N_{0}}q_{-N_{0}}w^*_{N_{0}}=q_{N-N_{0}},$ and moreover
$e_N=e_0.$ Therefore $\|ue_{n-1}u^*-e_{n}\|=0$ for $1\leq n \leq
N-N_0$, and also $ue_Nu^*=e_1,$ while for $N-N_0<n\leq N$ we have

$$\|ue_{n-1}u^*-e_n\|\leq2\|uw_{N-n+1}u^*-w_{N-n}\|<\varepsilon.$$
Also, clearly $e_{n}\in B_{\{y\}}$ for all $n$.

Set $e=\sum_{n=1}^{N}e_{n}$ and $p=1-e.$ We verify that $p$
satisfies (1) through (3).

First,
$$p-upu^*=ueu^*-e=\sum_{n=N_0+1}^N(ue_{n-1}u^*-e_n).$$
The terms in the sum are orthogonal and have norm less than
$\varepsilon$, so $\|upu^*-p\|<\varepsilon.$

Furthermore, $pup\in B_{\{y\}}.$

Next, let $g\in \mathcal{G}.$ The sets $U_0,U_1,\dots,U_N$ all have
diameter less than $\delta.$ We have $d(\sigma^N(y),y)<\delta,$ so
the choice of $\delta$ implies that
$d(\sigma^n(y),\sigma^{n-N}(y))<\delta_0$ for $N-N_0\leq n\leq N.$
Also, $U_{n-N}=\sigma^{n-N}(U_0)$ has diameter less than $\delta.$
Therefore $U_{n-N}\cup U_n$ has diameter less than
$2\delta+\delta_0\leq3\delta_0.$ Since $g$ varies by at most
$\frac{1}{4}\varepsilon$ on any set with diameter less than
$4\delta_0,$ and since the sets $
\sigma(Y),\sigma^2(Y),\sigma^{N-N_0-1}(Y),\sigma^{N-N_0}(Y)\cup
\sigma^{-N_0}(Y),\sigma^{N-N_0+1}(Y)\cup
\sigma^{-N_0+1}(Y),\\
\dots,\sigma^{N}(Y)\cup Y$ are disjoint.

For $0\leq n\leq N-N_0,$ we have $ge_n=gq_n=q_ng=e_ng.$

For $N-N_0<n\leq N$ and any $g\in \mathcal{G}$, we use $e_n\in
(q_{n-N}+q_n)B_{\{y\}}(q_{n-N}+q_n)$  and $(**)$ to get
$\|ge_n-e_ng\|=\|gw_k q_{-k}w_{k}^*-w_k
q_{-k}w_{k}^*g\|=\|w_{k}^*gw_k q_{-k}-
q_{-k}w_{k}^*gw_k\|<\varepsilon.$  It follows that
$\|pg-gp\|=\|ge-eg\|<\varepsilon.$ That $pgp\in B_{\{y\}}$ follows
from the fact that $g$ and $p$ are in this subalgebra. So we also
have (2) for $g$.

It remains only to verify (3). Let $\tau\in T(B_{\{y\}}),$ and let
$\mu$ be the corresponding $\sigma$-invariant probability measure on
$X$ by Lemma 3.3. We have

$$1-p=e\leq\sum_{n=-N_0}^N q_n,$$
so
$$\tau(1-p)\leq
\sum_{n=-N_0}^N\mu(\sigma^n(Y))=(N+N_0+1)\mu(Y)<\varepsilon.$$ This
completes the proof.
\end{proof}

We recall two results from Lemma 4.3 and Lemma 4.4 of \cite{H.Lin5}.

\begin{Lemma} Let $A$ be a  unital simple $C$*-algebra. Suppose that
for every $\varepsilon>0$ and every finite subset
$\mathcal{F}\subset A,$ there exists a unital $C$*-subalgebra
$B\subset A$ which has tracial rank zero and a projection $p\in B$
such that
$$\|pa-ap\|<\varepsilon \quad \mbox{and}\quad \mbox{dist}(pap,pBp)<\varepsilon$$
for all $a\in \mathcal{F}.$ Then $A$ has the local approximation
property of Popa \cite{S.Popa}, that is for every $\varepsilon>0$
and every finite subset $\mathcal{F}\subset A,$ there exists a
nonzero projection $q\in A$ and a finite dimensional unital
$C$*-subalgebra $D\subset qAq$ such that
$$\|qa-aq\|<\varepsilon\quad \mbox{and}\quad \mbox{dist}(qaq,D)<\varepsilon$$
for all $a\in \mathcal{F}.$
\end{Lemma}

\begin{Lemma} Let $A$ be a  unital simple $C$*-algebra. Suppose that
for every finite subset $\mathcal{F}\subset A,$ every
$\varepsilon>0,$ and every nonzero positive element $c\in A,$ there
exists a projection $p\in A$ and a unital simple subalgebra
$B\subset pAp$ with tracial rank zero such that:

(1) $\|[a,p]\|<\varepsilon$ for all $a\in \mathcal{F}.$

(2) dist$(pap,B)<\varepsilon$ for all $a\in \mathcal{F}.$

(3) $1-p$ is Murray-von Neumann equivalent to a projection in
$\overline{cAc}.$\\
Then $A$ has tracial rank zero.
\end{Lemma}

\begin{Theorem} Let $X$ be a Cantor set, and let $A$ be a
unital separable simple amenable $C$*-algebra with tracial rank zero which
satisfies the UCT. Let $C(X,A)$ denote all continuous functions from
$X$ to $A$ and  $\alpha$ be an automorphism of $C(X,A)$. Suppose
that $C(X,A)$ is $\alpha$-simple and $[\alpha]=[\mbox{id}_{1\otimes A}]$ in $KL(1\otimes A,1\otimes A)$. Then
$C(X,A)\rtimes_{\alpha}\mathbb{Z}$ is a unital simple $C$*-algebra
with tracial rank zero.
\end{Theorem}
\begin{proof} Let $B=C(X,A)\rtimes_{\alpha}\mathbb{Z},$ then
$B$ is unital simple by Lemma 2.4.

We verify the conditions of Lemma 3.6.

Let
$B_{\{y\}}=C^*(C(X,A),uC_0(X\backslash \{y\},A))$ for any given
point $y\in X.$ So $B_{\{y\}}$ has tracial rank zero by Lemma 3.2.
Thus, let $\mathcal{F}\subset B$ be a finite subset, let
$\varepsilon>0,$ and let $c\in B$ be a nonzero positive element.
Beyond Lemma 3.4, the main step of the proof is to find a nonzero
projection in $B_{\{y\}}$ which is Murray-von Neumann equivalent to
a projection in $\overline{cBc}.$

 The algebra $B$ is simple, so Lemma 3.4, Lemma3.5 and Lemma 2.12 of
 \cite{H.Lin6} imply that $B$ has property (SP). Therefore there is
 a nonzero projection $e\in \overline{cBc}.$\\
 Set
$$\delta_0=\frac{1}{18}\inf_{\tau\in T(B)}\tau(e)\leq \frac{1}{18}.$$
By Lemma 3.4, there is a projection $q\in B_{\{y\}}$ and an element
$b_0\in qB_{\{y\}}q$ such that
$$\|qe-eq\|<\delta_0,\quad \|qeq-b_0\|<\delta_0,\quad\mbox{and}\quad \sup_{\tau\in T(B_{\{y\}})}\tau(1-q)\leq \delta_0.$$
Then $\sup_{\tau\in T(B)}\tau(1-q)\leq \sup_{\tau\in
T(B_{\{y\}})}\tau(1-q)\leq \delta_0.$\\
Replacing $b_0$ by $\frac{1}{2}(b_0+b_0^*),$ we may assume that
$b_0$ is self-adjoint. We have $-\delta_0\leq b_0\leq 1+\delta_0,$ so
applying continuous functional calculus  we may find $b\in
qB_{\{y\}}q$ such that $0\leq b\leq1$ and $\|qeq-b\|<2\delta_0.$
Using $\|qe-eq\|<\delta_0$ on the last term in the second
expression, we get
$$\|b^2-b\|\leq3\|b-qeq\|+\|(qeq)^2-qeq\|<3\cdot2\delta_0+\delta_0=7\delta_0<\frac{1}{4}.$$
Therefore there is a projection $e_1\in qB_{\{y\}}q$ such that
$\|e_1-b\|<14\delta_0,$ giving $\|e_1-qeq\|<16\delta_0.$ Similarly (
actually, one gets a better estimate ) there is a projection $e_2\in
(1-q)B(1-q)$ such that $\|e_2-(1-q)e(1-q)\|<16\delta_0.$ Therefore
$\|e_1+e_2-[qeq+(1-q)e(1-q)]\|<16\delta_0$ and, using
$\|qe-eq\|<\delta_0$ again, we have $\|e_1+e_2-e\|<18\delta_0\leq1.$
It follows that $e_1\precsim e.$ Also, for $\tau\in T(B),$ we have
\begin{eqnarray*}\tau(e_1)&>&\tau(qeq)-16\delta_0=\tau(e)-\tau((1-q)e(1-q))-16\delta_0\\
&\geq&\tau(e)-\tau(1-q)-16\delta_0>0,\end{eqnarray*} so $e_1\neq 0.$

Now set $\varepsilon_0=\inf(\{\tau(e_1): \tau\in T(B_{\{y\}})\}).$
By Lemma 3.4, there is a projection $p\in B_{\{y\}}$ such that:

(1) $\|pa-ap\|<\varepsilon$ for all $a\in \mathcal{F}.$

(2) $pap\in pB_{\{y\}}p$ for all $a\in \mathcal{F}.$

(3) $\tau(1-p)<\varepsilon_0$ for all $\tau \in T(B_{\{y\}}).$\\
Since $B_{\{y\}}$ has tracial rank zero which implies that the order
on projections is determined by traces, it follows that $1-p\precsim
e_1 \precsim e.$ Since $pB_{\{y\}}p$ also has tracial rank zero, we
have verified the hypotheses of Lemma 3.6. Thus $B$ has tracial rank
zero.

\end{proof}

\begin{Cor} For j=1,2, let $X$ be a Cantor set, and let $A_j$ be a
unital separable simple amenable $C$*-algebra with tracial rank zero which
satisfy the UCT. Let $C(X,A_j)$ denote all continuous functions from
$X$ to $A_j$, and let $\alpha_j$ be an automorphism of $C(X,A_j)$. Suppose
that $C(X,A_j)$ is $\alpha_j$-simple and $[\alpha_j]=[\mbox{id}_{1\otimes A_j}]$ in $KL(1\otimes A_j,1\otimes A_j)$. Let
$B_j=C(X,A_j)\rtimes_{\alpha_j}\mathbb{Z}$ for $j=1,2.$ Then $B_1\cong
B_2$ if and only if
$$(K_0(B_1), K_0(B_1)_+,[1_{B_1}],K_1(B_1))\cong (K_0(B_2), K_0(B_2)_+,[1_{B_2}],K_1(B_2)).$$
\end{Cor}
\begin{proof} This is immediate from Theorem 5.2 of \cite{H.Lin7}
and Theorem 3.7.
\end{proof}

\

\

{\scshape Department of Mathematics, Tongji University,
Shanghai 200092, P.R.CHINA.}

{\it E-mail address}: huajiajie2006@hotmail.com
\end{document}